\documentclass[12pt]{amsproc}

\usepackage{amssymb,amsmath}

\newtheorem{theorem}{Theorem}[section]
\newtheorem{lemma}[theorem]{Lemma}

\theoremstyle{definition}
\newtheorem{definition}[theorem]{Definition}

\newtheorem{corollary}[theorem]{Corollary}
\newtheorem{proposition}[theorem]{Proposition}

\theoremstyle{remark}

\numberwithin{equation}{section}

\begin{document}

\title{Minimal free resolution of the associated graded ring of certain  monomial curves}

\author{P{\i}nar METE}
\address{Department of Mathematics,
Bal{\i}kesir University, Bal{\i}kesir , 10145 Turkey}
\email{pinarm@balikesir.edu.tr}
\author{Esra Emine ZENG\.{I}N}
\address{Department of Mathematics,
Bal{\i}kesir University, Bal{\i}kesir , 10145 Turkey}
\email{esrazengin103@gmail.com}

\date{\today}

\begin{abstract}
In this article, we give the explicit minimal free resolution of  the associated graded ring of certain affine monomial curves in affine 4-space based on the standard basis theory. As a result, we give the minimal graded free resolution and compute the Hilbert function of the tangent cone of these families.
\end{abstract}

\maketitle
\section{Introduction}
In this article, we study the minimal free resolution of the associated graded ring of the local ring $A$ of a monomial curve
$C \subset \mathbb{A}^{4}$ corresponding to an arithmetic sequence based on the standard basis theory.
The associated graded ring $G=gr_{m}(A)= \bigoplus_{i=0}^{\infty} (m^{i} / m^{i+1})$ of $A$ with maximal ideal $m$ is a standard graded $k$-algebra. Since it corresponds to the important geometric construction, it has been studied  to get comprehensive information on the local ring (see \cite{rossi-valla,rossi-sharifan,molinelli-tamone1,molinelli-tamone2,molinelli-patil-tamone}).
Because the minimal finite free resolution of a finitely generated $k$-algebra is a very useful tool to extract information about the algebra, finding an explicit minimal free resolution of a standard $k$-algebra  is a basic problem. This difficult problem has been extensively studied in the case of affine monomial curves \cite{sharifan-nahandi,sengupta,gimenez-sengupta-srinivasan,oneto-tamone,barucci-froberg-sahin}.

We recall that  a monomial affine curve $C$ has a parametrization
\[x_0=t^{m_0}, \; x_1=t^{m_1}, \;\ldots, \; x_n=t^{m_n}\]
where $m_0,m_1,\ldots,m_n$ are positive integers with
$gcd(m_0,m_1,...,m_n)=1$. The additive semigroup, which is denoted by
$$<m_0,m_1,...,m_n>=\{\; \sum_{0 \leq i \leq n} \mathbb{N}m_i \; \mid \; \mathbb{N}=
\{0,1,2,\ldots \} \}$$
\noindent generated minimally by $m_0,m_1,...,m_n$, i.e., $m_j \notin \sum_{0 \leq i \leq n; i \neq j} \mathbb{N}m_i$ for $i \in \{0,\ldots,n\}$.

Assume that $m_0,m_1,\ldots,m_n$ be positive integers such that $0 < m_0 < m_1 < \ldots < m_n$ and $m_i=m_0 + id$ for 
every $1 \leq i \leq n$, where $d$ is the common difference, i.e. the integers $m_i$'s form an arithmetic progression.
The monomial curve which is defined parametrically by 
\[x_0=t^{m_0}, \; x_1=t^{m_1}, \;\ldots, \; x_n=t^{m_n}\]
\noindent such that $0 < m_0 < m_1 < \ldots < m_n$ form an arithmetic progression is called a certain monomial curve.

In order to study the associated graded ring 
of a monomial curve $C$ at the origin, it is possible to consider
either the associated graded ring of
$A=k[[t^{m_0},t^{m_1},...,t^{m_n}]]$ with respect to the maximal
ideal $m=(t^{m_0},t^{m_1},...,t^{m_n})$ which is denoted by
$gr_m(k[[t^{m_0},t^{m_1},...,t^{m_n}]])$, or the ring
$k[x_0,x_1,...,x_n]/I(C)^{*}$, where $I(C)^{*}$ is the ideal generated by the polynomials $f^{*}$
for $f$ in $I(C)$, where $f^{*}$ is the homogeneous summand of $f$
of the least degree,  since they are isomorphic.  We recall that $I(C)^{*}$ is the defining ideal of the tangent cone of the curve $C$ at the origin.

Our main aim in this paper is to give  an explicit minimal free resolution of the associated graded ring for certain monomial curves in affine 4-space. Even if one can obtain the numerical invariants of the minimal free resolution of the tangent cone of certain monomial curves in $\mathbb{A}^{4}$  by using the Theorem 4.1 and Proposition 4.6 in \cite{sharifan-nahandi}, we give the  minimal free resolution of the tangent cone of certain monomial curves in affine 4-space in an explicit form by giving a new proof based on the standard basis theory. 
Using the standard basis theory and knowing the minimal generating set  of binomial generators of the defining ideal of certain monomial curve  explicitly from  \cite{patil}, we find the minimal generators of the tangent cone of a certain monomial curve in affine 4-space. By knowing the minimal generators, we show the Cohen-Macaulayness of the tangent cone of these families of curves. We obtain explicit minimal free resolution by using Schreyer's theorem but prove it using the Buchsbaum-Eisenbud theorem \cite{eisenbud-buchsbaum}.
Finally, we give the minimal graded free resolutions and as a corollary compute the Hilbert function of the tangent cones for these families. All computations have been carried out using {\footnotesize SINGULAR}\cite{singular}.

\section{Minimal generators of the associated graded ring}

In this section, we find the minimal generators of the tangent
cone of the certain  monomial curve $C$ having the defining ideal as in Theorem
4.5 in \cite{patil} in affine 4-space. First, we recall the theorem which gives the construction
of the minimal set of generators for the defining ideal of  certain affine monomial curve in 
$\mathbb{A}^{4}$ .

Let $m_0 < m_1 < m_2 < m_3$ be positive integers with $gcd(m_0,m_1,m_2,m_3)=1$ and
assume that  $m_0,m_1,m_2,m_3$ form an arithmetic progression with common difference $d$. 
Let $R=k[x_0,x_1,x_2,y]$ be a polynomial ring over the field $k$.  We use $y$ instead of $x_3$ by
following the same notation in \cite{sengupta,sengupta2,patil}. Let $\phi : R \rightarrow k[t^{m_0},t^{m_1},t^{m_2},t^{m_3}]$ be the $k$-algebra homomorphism defined by
\begin{center}
$\phi(x_0)=t^{m_0}$, $\phi(x_1)=t^{m_1}$, $\phi(x_2)=t^{m_2}$, $\phi(y)=t^{m_3}$ 
\end{center}
and $I(C)=Ker (\phi)$. Let us write $m_{0}=3a+b$ such that $a$ and $b$ are positive integers $a \geq 1$ and $b \in [1,3]$.
In \cite{sengupta},  the following theorem is given as a definition.

\begin{theorem} \label{thm1} {\rm \cite{patil}} Let \\[1mm]
\indent\hspace{1cm} $\xi_{11} := x_1^{2}-x_{0}x_{2},$\\[1mm]
\indent\hspace{1cm} $\varphi_{i} := x_{i+1}x_2-x_{i}y,\; for \; 0 \leq i \leq 1.$\\[1mm]
\indent\hspace{1cm} $\psi_{j} := x_{b+j}y^{a}-x_{0}^{a+d}x_{j},\; if \; 1 \leq b \leq 2 \;\; and \; \; 0 \leq j \leq 2-b.$\\[1mm]
\indent\hspace{1cm} $\theta := y^{a+1}-x_{0}^{a+d}x_{3-b}$. \\[1mm]
\indent\hspace{1cm} $\displaystyle
G :=\left\{
\begin{array}{lr}
\{\xi_{11} \}  \cup \{ \varphi_{0}, \varphi_{1}\} \cup \{ \psi_{0}, \psi_{1}\} \cup \{ \theta\}& ~ if\;\;\; b=1, \\
\{\xi_{11} \}  \cup \{ \varphi_{0}, \varphi_{1}\} \cup \{ \psi_{0}\} \cup \{ \theta\}& ~ if\;\;\; b=2, \\
\{\xi_{11} \}  \cup \{ \varphi_{0}, \varphi_{1}\}  \cup \{ \theta\}& ~ if\;\;\; b=3.
\end{array}
\right.
$\\[1mm]

\noindent then, G is a minimal generating set for the defining ideal  $I(C)$.
\end{theorem}

Now, we recall  the definition of the negative degree reverse lexicographical
ordering among the other local orderings.

\begin{definition} {\rm \cite[p.14]{greuel-pfister}} (negative degree reverse
lexicographical ordering)
\begin{center}
$x^{\alpha} >_{ds} x^{\beta}:\Leftrightarrow$ $degx^{\alpha}<degx^{\beta}$,
where $degx^{\alpha}=\alpha_{1}+...+\alpha_{n}$, \\
or $(degx^{\alpha}=degx^{\beta}$ and $\exists 1\leq i\leq n:
\alpha_{n}=\beta_{n},...,\alpha_{i+1}=\beta_{i+1},\alpha_{i}<\beta_{i}).$
\end{center}
\end{definition}

In the following Lemma, we show that the above set G is also standard basis with respect to $>_{ds}$.

\begin{lemma}\label{lemma1}
The minimal set $G$ is a standard basis with
respect to the negative degree reverse lexicographical ordering
$>_{ds}$ with $x_0 > x_1 > x_2 > y$.
\end{lemma}

\noindent {\em Proof.}
We apply the standard basis algorithm 
to the set $G$. We will prove for $b=1,2$, and $3$, respectively.
By using the notation in \cite{greuel-pfister}, we denote the leading monomial of a polynomial $f$
by $LM(f)$, the S-polynomial of the polynomials $f$ and $g$ by $spoly(f,g)$ and the
Mora's polynomial weak normal form of $f$ with respect to G by $NF(f  \mid  G)$.\\

\noindent {\bf Case \textit{b} $=$ 1.}\\
From the minimal generating set $G$ in Theorem~\ref{thm1}, we obtain
$$G = \big\{\xi_{11} = x_1^{2}-x_{0}x_{2},\;\; \varphi_{0} = x_{1}x_2-x_{0}y,\;\; \varphi_{1} = x_{2}^2-x_{1}y,$$
$$\psi_{0} = x_{1}y^{a}-x_{0}^{a+d+1},\;\; \psi_{1} = x_{2}y^{a}-x_{0}^{a+d}x_{1},\;\; \theta = y^{a+1}-x_{0}^{a+d}x_{2}\big\}.$$

\noindent Recalling that the ordering is the negative degree reverse lexicographical ordering, we have
$LM(\xi_{11}) = x_1^{2}$, $LM(\varphi_{0})$ = $x_{1}x_{2}$, $LM(\varphi_{1}) = x_2^{2}$,
$LM(\psi_{0}) = x_{1}y^{a}$, $LM(\psi_{1})$ = $x_{2}y^{a}$ and $LM( \theta) = y^{a+1}$.

We begin with $\xi_{11}$ and $\varphi_{0}$. $LM(\xi_{11})$ = $x_1^{2}$ and
$LM(\varphi_{0}) = x_{1}x_{2}$. We compute
${\rm spoly}(\xi_{11}, \varphi_{0})$ = $x_{0}x_{1}y-x_{0}x_{2}^2$. 
$LM({\rm spoly}(\xi_{11}, \varphi_{0}))$ = $x_{0}x_{2}^2$. Among the leading monomials of the elements of $G$,
only $LM(\varphi_{1})$ divides $LM({\rm spoly}(\xi_{11}, \varphi_{0}))$.
Also ecart$(\varphi_{1})$=ecart$({\rm spoly}(\xi_{11}, \varphi_{0}))$ = 0.
${\rm spoly}(\varphi_{1}, {\rm spoly}(\xi_{11}, \varphi_{0}))$ = 0 implies
$NF({\rm spoly}(\xi_{11},\varphi_{0})\,\vert\,G)$ = 0. 

Next, we choose $\xi_{11}$ and $\varphi_{1}$. Since ${\rm lcm}(LM(\xi_{11}), LM(\varphi_{1}))$
= $LM(\xi_{11}).LM(\varphi_{1})$, then $NF({\rm spoly}(\xi_{11}, \varphi_{1})\,\vert\,\{\xi_{11}, \varphi_{1}\})$ = 0.
This implies that $NF({\rm spoly}(\xi_{11}, \varphi_{1})\,\vert\,G)$ = 0. 
In the same manner, $NF({\rm spoly}(\xi_{11}, \psi_{1})\,\vert\,G)$ = 0,
$NF({\rm spoly}(\xi_{11}, \theta)\,\vert\,G)$ = 0, $NF({\rm spoly}(\varphi_{0},\theta)\,\vert\,G)$ = 0, 
$NF({\rm spoly}(\varphi_{1},\psi_{0})\,\vert\,G)$ = 0
 and
$NF({\rm spoly}(\varphi_{1},\theta)\,\vert\,G)$ = 0.

Now, we compute S-polynomial of $\xi_{11}$ and $\psi_{0}$. 
${\rm spoly}(\xi_{11}, \psi_{0}) = x_{0}^{a+d+1}x_{1}-x_{0}x_{2}y^{a}$. 
Among the leading monomials of the elements of $G$,
only $LM(\psi_{1})$ divides $LM({\rm spoly}(\xi_{11}, \psi_{0}))$ = $x_{0}x_{2}y^{a}$.
Also ecart$(\psi_{1})$=ecart$({\rm spoly}(\xi_{11}, \psi_{0}))$ = $d$.
${\rm spoly}(\psi_{1}, {\rm spoly}(\xi_{11}, \psi_{0}))$ =0 implies
$NF({\rm spoly}(\xi_{11},\psi_{0})\,\vert\,G)$ = 0. 

Again, we compute S-polynomial of $\varphi_{0}$ and $\varphi_{1}$. 
${\rm spoly}(\varphi_{0}, \varphi_{1})$ = $x_{1}^{2}y-x_{0}x_{2}y$. 
Among the leading monomials of the elements of $G$,
only $LM(\xi_{11})$ divides\\ 
$LM({\rm spoly}(\varphi_{0}, \varphi_{1}))$ = $x_{1}^2y$.
Also ecart$(\xi_{11}) = $ecart$({\rm spoly}(\varphi_{0}, \varphi_{1}))$ = 0.
${\rm spoly}(\xi_{11}, {\rm spoly}(\varphi_{0}, \varphi_{1}))$ = 0 implies
$NF({\rm spoly}(\varphi_{0},\varphi_{1})\,\vert\,G)$ = 0. 

Now choose $\varphi_{0}$ and $\psi_{0}$. Then, S-polynomial of $\varphi_{0}$ and $\psi_{0}$ is
${\rm spoly}(\varphi_{0}, \psi_{0})=x_{0}^{a+d+1}x_{2}-x_{0}y^{a+1}$. 
Once again, only $LM(\theta)$ divides $LM({\rm spoly}(\varphi_{0}, \psi_{0}))$ = $x_{0}y^{a+1}$
among the leading monomials of the elements of $G$. 
Also ecart$(\theta)$ = ecart$({\rm spoly}(\varphi_{0}, \psi_{0}))$ = $d$.
${\rm spoly}(\theta, {\rm spoly}(\varphi_{0}, \psi_{0}))$ = 0 implies 
$NF({\rm spoly}(\varphi_{0},\psi_{0})\,\vert\,G)$ = 0. 

Similarly, ${\rm spoly}(\varphi_{0}, \psi_{1}) = x_{0}^{a+d}x_{1}^{2}-x_{0}y^{a+1}$. 
Again, as in the previous case $LM(\theta)$ divides $LM({\rm spoly}(\varphi_{0}, \psi_{1}))$ = $x_{0}y^{a+1}$.
Also ecart$(\theta)$ = ecart$({\rm spoly}(\varphi_{0}, \psi_{1}))$ = $d$.
${\rm spoly}(\theta, {\rm spoly}(\varphi_{0}, \psi_{1}))$ = $x_{0}^{a+d}x_{1}^{2}-x_{0}^{a+d+1}x_2.$ 
Among the leading monomials of the elements of $G$,
only $LM(\xi_{11}) = x_{1}^2$ divides $LM({\rm spoly}({\rm spoly}(\theta, {\rm spoly}(\varphi_{0}, \psi_{1}))))$ = $x_{0}^{a+d}x_{1}^2$.
ecart$(\xi_{11})$ = ecart$({\rm spoly}(\theta, {\rm spoly}(\varphi_{0}, \psi_{1})))$ = 0.
${\rm spoly}(\xi_{11}, {\rm spoly}(\theta, {\rm spoly}(\varphi_{0}, \psi_{1})))$ = 0 implies
$NF({\rm spoly}(\varphi_{0},\psi_{1})\,\vert\,G)$ = 0. 

Similarly, we compute ${\rm spoly}(\varphi_{1}, \psi_{1})$ = $x_{0}^{a+d}x_{1}x_{2}-x_{1}y^{a+1}$.
Among the leading monomials of the elements of $G$,
only $LM(\psi_{0})$  and $LM(\theta)$ divides $LM({\rm spoly}(\varphi_{1}, \psi_{1}))$ = $x_{1}y^{a+1}$.
Note that ecart$(\psi_{0})$ = ecart$(\theta) = d$.
Firstly, beginning with $\psi_{0}$,
${\rm spoly}(\psi_{0}, {\rm spoly}(\varphi_{1}, \psi_{1}))$ = $x_{0}^{a+d}x_{1}x_{2}-x_{0}^{a+d+1}y$.
Among the leading monomials of the elements of $G$, $LM(\varphi_{1}) = x_{1}x_{2}$ divides
$LM({\rm spoly}(\psi_{0}, {\rm spoly}(\varphi_{1}, \psi_{1})))$.
Also ecart$(\varphi_{1})$ = ecart$({\rm spoly}(\psi_{0}, {\rm spoly}(\varphi_{1}, \psi_{1})))$ = 0.
${\rm spoly}(\varphi_{1}, {\rm spoly}(\psi_{0}, {\rm spoly}(\varphi_{1}, \psi_{1})))$ = 0.
Secondly, taking $\theta$, ${\rm spoly}(\theta, {\rm spoly}(\varphi_{1}, \psi_{1}))$ = 0.
Thus, $NF({\rm spoly}(\varphi_{1},\psi_{1})\,\vert\,G)$ = 0.

We continue by computing ${\rm spoly}(\psi_{0}, \psi_{1})=x_{0}^{a+d}x_{1}^2-x_{0}^{a+d+1}x_{2}$.
$LM(\xi_{11}) = x_{1}^{2}$ divides $LM({\rm spoly}(\psi_{0}, \psi_{1}))$ = $x_{0}^{a+d}x_{1}^{2}.$
Also ecart$(\xi_{11})$ = ecart$({\rm spoly}(\psi_{0}, \psi_{1}))$ = 0.
${\rm spoly}(\xi_{11}, {\rm spoly}(\psi_{0}, \psi_{1})$ = 0 implies
$NF({\rm spoly}(\psi_{0},\psi_{1})\,\vert\,G)$ = 0. 

In the same manner, ${\rm spoly}(\psi_{0}, \theta)$ = $x_{0}^{a+d}x_{1}x_{2}-x_{0}^{a+d+1}y$.
$LM(\varphi_{0}) = x_{1}x_{2}$ divides $LM({\rm spoly}(\psi_{0}, \theta))$ = $x_{0}^{a+d}x_{1}x_{2}.$
Also ecart$(\varphi_{0})$ = ecart$({\rm spoly}(\psi_{0}, \theta))$ = 0.
${\rm spoly}(\varphi_{0}, {\rm spoly}(\psi_{0}, \theta))$ = 0 implies
$NF({\rm spoly}(\psi_{0},\theta)\,\vert\,G)$ = 0. 

Finally, we compute ${\rm spoly}(\psi_{1}, \theta)$ = $x_{0}^{a+d}x_{2}^{2}-x_{0}^{a+d}x_{1}y$.
$LM(\varphi_{1})$ = $x_{2}^{2}$ divides $LM({\rm spoly}(\psi_{1}, \theta))$ = $x_{0}^{a+d}x_{2}^{2}.$
Also ecart$(\varphi_{1})$ = ecart$({\rm spoly}(\psi_{1}, \theta))$ = 0.\\
${\rm spoly}(\varphi_{1}, {\rm spoly}(\psi_{1}, \theta))$ = 0 implies
$NF({\rm spoly}(\psi_{0},\theta)\,\vert\,G)$ = 0.

\vspace{0.5cm}
\noindent {\bf Case \textit{b} $=$ 2.}\\
As in the previous case, we obtain by the minimal generating set $G$ in Theorem~\ref{thm1}, 
$$G=\big\{\xi_{11} = x_1^{2}-x_{0}x_{2},\;\; \varphi_{0} = x_{1}x_2-x_{0}y,\;\; \varphi_{1} = x_{2}^2-x_{1}y,$$
$$\psi_{0} = x_{2}y^{a}-x_{0}^{a+d+1},\;\;  \theta = y^{a+1}-x_{0}^{a+d}x_{1}\big\}.$$ 
$LM(\xi_{11}) = x_1^{2}$, $LM(\varphi_{0}) = x_{1}x_{2}$, $LM(\varphi_{1}) = x_{2}^{2}$,
$LM(\psi_{0}) = x_{2}y^{a}$ and $LM( \theta) = y^{a+1}$ with respect to
the negative degree reverse lexicographical ordering.

We begin with $\xi_{11}$ and $\varphi_{0}$. This case is exactly the same as in $b=1$.

Next, we choose $\xi_{11}$ and $\varphi_{1}$.  As in the first case, since 
${\rm lcm}(LM(\xi_{11}),LM(\varphi_{1}))$ = $LM(\xi_{11}).LM(\varphi_{1})$, then $NF({\rm spoly}(\xi_{11},\varphi_{1})\,\vert\,\{\xi_{11},\varphi_{1}\})$ = 0.
Therefore, this implies that $NF({\rm spoly}(\xi_{11},\varphi_{1})\,\vert\,G)$ = 0. 
In the same manner, $NF({\rm spoly}(\xi_{11}, \psi_{0})\,\vert\,G)$ = 0,
$NF({\rm spoly}(\xi_{11}, \theta)\,\vert\,G)$ = 0, $NF({\rm spoly}(\varphi_{0},\theta)\,\vert\,G)$ = 0 and
$NF({\rm spoly}(\varphi_{1},\theta)\,\vert\,G)$ =0.

Again, we compute S-polynomial of $\varphi_{0}$ and $\varphi_{1}$. 
This one is also the same as in the previous case.

Now choose $\varphi_{0}$ and $\psi_{0}$. Then, S-polynomial of $\varphi_{0}$ and $\psi_{0}$ is
${\rm spoly}(\varphi_{0}, \psi_{0})$ = $x_{0}^{a+d+1}x_{1}-x_{0}y^{a+1}$. 
Once again, only $LM(\theta) = y^{a+1}$ divides $LM({\rm spoly}(\varphi_{0}, \psi_{0}))$ = $x_{0}y^{a+1}$
among the leading monomials of the elements of $G$. 
Also, ecart$(\theta)$ = ecart$({\rm spoly}(\varphi_{0}, \psi_{0}))$ = $d$.
${\rm spoly}(\theta, {\rm spoly}(\varphi_{0}, \psi_{0}))$ = 0 implies
$NF({\rm spoly}(\varphi_{0},\psi_{0})\,\vert\,G)$ = 0. 

Similarly, we compute ${\rm spoly}(\varphi_{1}, \psi_{0})$ = $x_{0}^{a+d+1}x_{2}-x_{1}y^{a+1}$.
Among the leading monomials of the elements of $G$,
only  $LM(\theta)$ divides $LM({\rm spoly}(\varphi_{1}, \psi_{0}))$ = $x_{1}y^{a+1}$.
ecart$(\theta)$ = ecart$({\rm spoly}(\varphi_{1}, \psi_{0}))$ = $d$.
${\rm spoly}(\theta, {\rm spoly}(\varphi_{1}, \psi_{0}))$ = $x_{0}^{a+d+1}x_{2}-x_{0}^{a+d}x_{1}^{2}$.
Since ${\rm spoly}(\theta, {\rm spoly}(\varphi_{1}, \psi_{0}))$ is not zero, again among the leading 
monomials of the elements of $G$, $LM(\xi_{11})$ = $x_{1}^{2}$ divides 
$LM({\rm spoly}(\theta, {\rm spoly}(\varphi_{1}, \psi_{0})))$ = $x_{0}^{a+d}x_{1}^{2}$.
ecart$(\xi_{11})$ = ecart$({\rm spoly}(\theta, {\rm spoly}(\varphi_{1}, \psi_{0})))$ = 0.\\
${\rm spoly}(\xi_{11}, {\rm spoly}(\theta, {\rm spoly}(\varphi_{1}, \psi_{0})))$ = 0.
Thus, $NF({\rm spoly}(\varphi_{1},\psi_{0})\,\vert\,G)$ = 0.

Finally, we compute ${\rm spoly}(\psi_{0}, \theta)$ = $x_{0}^{a+d}x_{1}x_{2}-x_{0}^{a+d+1}y$.
$LM(\varphi_{0}) = x_{1}x_{2}$ divides $LM({\rm spoly}(\psi_{0}, \theta))$ = $x_{0}^{a+d}x_{1}x_{2}.$
Also ecart$(\varphi_{0})$ = ecart$({\rm spoly}(\psi_{0}, \theta))$ = 0.
${\rm spoly}(\varphi_{0}, {\rm spoly}(\psi_{0}, \theta))$ = 0 implies
$NF({\rm spoly}(\psi_{0},\theta)\,\vert\,G)$ = 0. 

\vspace{0.5cm}
\noindent {\bf Case \textit{b} $=$ 3.}\\
Finally, by writing 3 instead of b in the minimal generating set $G$ in Theorem~\ref{thm1}, we obtain
$$G=\big\{\xi_{11} = x_1^{2}-x_{0}x_{2},\;\; \varphi_{0} = x_{1}x_2-x_{0}y,\;\; \varphi_{1} = x_{2}^2-x_{1}y,\;\;  
\theta = y^{a+1}-x_{0}^{a+d+1}\big\}.$$ 
In the same manner, $LM(\xi_{11}) = x_1^{2}$, $LM(\varphi_{0})$ = $x_{1}x_{2}$, 
$LM(\varphi_{1}) = x_{2}^{2}$ and $LM( \theta)$ = $y^{a+1}$ with respect to
the negative degree reverse lexicographical ordering $>_{ds}$.

As in the previous cases, we begin with $\xi_{11}$ and $\varphi_{0}$ and this case is exactly the same as in $b=1$.

In the same manner,  $NF({\rm spoly}(\{\xi_{11},\varphi_{1})\,\vert\,G)$ = 0,
$NF({\rm spoly}(\xi_{11}, \theta)\,\vert\,G)$ = 0, $NF({\rm spoly}(\varphi_{0},\theta)\,\vert\,G)$ = 0 and
$NF({\rm spoly}(\varphi_{1},\theta)\,\vert\,G)$ = 0.

Finally, the computation of the S-polynomial of $\varphi_{0}$ and $\varphi_{1}$  also results 
as in the case $b=1$.

Therefore, if $b=$1,2 and 3, we conclude that the set $G$ is a standard basis
with respect to the negative degree reverse lexicographical ordering $>_{ds}$.

\begin{flushright}
$\Box$
\end{flushright}

We can now find the minimal generating set of the tangent cone by using the above lemma.

\begin{proposition}\label{prop1} Let $C$ be a certain monomial
curve having parametrization 
\[x_0=t^{m_0}, \; x_1=t^{m_1}, \; x_2=t^{m_2}, \; y=t^{m_3}\]
\noindent  $m_0=3a+b$ for positive integers $a \geq 1$ and $b \in [1,3]$
and $0 < m_0 < m_1 < m_2 < m_3$ form an arithmetic progression with common difference $d$
and let the generators of  the defining ideal $I(C)$ be given by the set $G$ in Theorem \ref{thm1}.
Then the defining ideal $I(C)^{*}$ of the tangent cone is generated by the set $G^{*}$ consisting of
the least homogeneous summands of the binomials in $G$.

\end{proposition}

\noindent {\em Proof}. By the
Lemma~\ref{lemma1}, 
\[
G :=\left\{
\begin{array}{lr}
\{\xi_{11} \}  \cup \{ \varphi_{0}, \varphi_{1}\} \cup \{ \psi_{0}, \psi_{1}\} \cup \{ \theta\}& ~ if\;\;\; b=1, \\
\{\xi_{11} \}  \cup \{ \varphi_{0}, \varphi_{1}\} \cup \{ \psi_{0}\} \cup \{ \theta\}& ~ if\;\;\; b=2, \\
\{\xi_{11} \}  \cup \{ \varphi_{0}, \varphi_{1}\}  \cup \{ \theta\}& ~ if\;\;\; b=3.

\end{array}
\right.
\]
as in Theorem \ref{thm1}, is a standard basis of $I(C)$ with
respect to a local degree ordering $>_{ds}$ with respect to 
$x_0 > x_1 > x_2 > y$. Then, from \cite [Lemma 5.5.11]{greuel-pfister},
$I(C)^{*}$ is generated by the least homogeneous summands of the elements in the standard basis.
Thus, $I(C)^*$ is generated by\\
if $b=1$
$$G^{*}=\big\{\xi_{11}^{*} = x_1^{2}-x_{0}x_{2},\;\; \varphi_{0}^{*}= x_{1}x_2-x_{0}y,\;\; \varphi_{1}^{*}= x_{2}^2-x_{1}y,\;\; \psi_{0}^{*} = x_{1}y^{a}, $$ 
$$\psi_{1}^{*}= x_{2}y^{a},\;\; \theta^{*} = y^{a+1}\big\},$$
if $b=2$
$$G^{*}=\big\{\xi_{11} = x_1^{2}-x_{0}x_{2},\;\; \varphi_{0}^{*} = x_{1}x_2-x_{0}y,\;\; 
\varphi_{1} ^{*}= x_{2}^2-x_{1}y,\;\; \psi_{0}^{*} = x_{2}y^{a},  \theta^{*} = y^{a+1}\big\},$$ 
and if $b=3$
$$G^{*}=\big\{\xi_{11}^{*} = x_1^{2}-x_{0}x_{2},\;\; \varphi_{0}^{*} = x_{1}x_2-x_{0}y,\;\; 
\varphi_{1} ^{*}= x_{2}^2-x_{1}y,\;\;  \theta^{*} = y^{a+1}\big\}.$$ 

\begin{flushright}
$\Box$
\end{flushright}

\begin{theorem}\label{thm2} Let $C$ be a certain monomial
curve having parametrization 
\[x_0=t^{m_0}, \; x_1=t^{m_1}, \; x_2=t^{m_2}, \; y=t^{m_3}\]
\noindent  $m_0=3a+b$ for positive integers $a \geq 1$ and $b \in [1,3]$
and $0 < m_0 < m_1 < m_2 < m_3$ form an arithmetic progression with common difference $d$.
The certain monomial curve $C$ with the defining ideal $I(C)$ as in Theorem \ref{thm1}  has
Cohen-Macaulay tangent cone at the origin. 
\end{theorem}
\begin{proof}
We  can apply the Theorem 2.1 in \cite{arslan1} to the generators of the
tangent cone which are given by the set\\
if $b=1$
$$G^{*}=\big\{\xi_{11}^{*} = x_1^{2}-x_{0}x_{2},\;\; \varphi_{0}^{*}= x_{1}x_2-x_{0}y,\;\; 
\varphi_{1}^{*}= x_{2}^2-x_{1}y,\;\; \psi_{0}^{*} = x_{1}y^{a}, $$ 
$$\psi_{1}^{*}= x_{2}y^{a},\;\; \theta^{*} = y^{a+1}\big\},$$
if $b=2$
$$G^{*}=\big\{\xi_{11} = x_1^{2}-x_{0}x_{2},\;\; \varphi_{0}^{*} = x_{1}x_2-x_{0}y,\;\;
\varphi_{1} ^{*}= x_{2}^2-x_{1}y,\;\; \psi_{0}^{*} = x_{2}y^{a},  \theta^{*} = y^{a+1}\big\},$$ 
and if $b=3$
$$G^{*}=\big\{\xi_{11}^{*} = x_1^{2}-x_{0}x_{2},\;\; \varphi_{0}^{*} = x_{1}x_2-x_{0}y,\;\; 
\varphi_{1} ^{*}= x_{2}^2-x_{1}y,\;\;  \theta^{*} = y^{a+1}\big\}.$$ 
All of these sets are
Gr\"{o}bner bases with respect to the reverse lexicographic order
with $x_0>y>x_1>x_2$. Since $x_0$ does not divide the leading
monomial of any element in $G^{*}$ in all three cases, the ring
$k[x_0,x_1,x_2,y]/I(C)^*$ is Cohen-Macaulay from Theorem 2.1
in \cite{arslan1}. Thus, $R={\rm gr}_m(k[[t^{m_0},t^{m_1},t^{m_2},t^{m_3}]])\cong
k[x_0,x_1,x_2,y]/I(C)^{*}$ is Cohen-Macaulay.
\end{proof}

\section{Minimal free resolution of the associated graded ring}
In this section, we study the minimal free resolution of  
${\rm gr}_m(k[[t^{m_0},t^{m_1},t^{m_2},t^{m_3}]])$ of the certain monomial curve $C$ in affine 4-space.

\begin{theorem}\label{thm3} 
Let $C$ be a certain affine monomial curve in $\mathbb{A}^{4}$ 
having parametrization 
\[x_0=t^{m_0}, \; x_1=t^{m_1}, \; x_2=t^{m_2}, \; y=t^{m_3}\]
\noindent  $m_0=3a+b$ for positive integers $a \geq 1$ and $b \in [1,3]$
and $0 < m_0 < m_1 < m_2 < m_3$ form an arithmetic progression with common difference $d$.
Then the sequence of R-modules
\begin{center}
$0 \longrightarrow R^{\beta_{3}(b)} \xrightarrow{\phi_{3}(b)} R^ {\beta_{2}(b)} \xrightarrow{\phi_{2}(b)}  R^ {\beta_{1}(b)} \xrightarrow{\phi_{1}(b)} R \xrightarrow{\phi}\ G \longrightarrow 0$
\end{center}
is a minimal free resolution for the tangent cone of $C$, where

\begin{eqnarray}
\nonumber
\beta_{1}(b) =\left\{
\begin{array}{ll}
6& if\;\; b=1, \\
5& if\;\; b=2, \\
4& if\;\; b=3,
\end{array}
\right.
,& 
\beta_{2}(b) =\left\{
\begin{array}{ll}
8& if\;\; b=1, \\
5& if\;\; b=2, \\
5& if\;\; b=3,
\end{array}
\right.
,&
\beta_{3}(b) =\left\{
\begin{array}{ll}
3& if\;\; b=1, \\
1& if\;\; b=2, \\
2& if\;\; b=3.
\end{array}
\right.
\end{eqnarray}

\noindent and $\phi$'s denote the canonical surjections and the maps between R-modules depend on b 

{\footnotesize
$$
\phi_{1}(b=1)=\Big(
\begin{matrix}
g_1=x_{1}^{2}-x_{0}x_{2} &
g_2=x_{1}x_{2}-x_{0}y &
g_3=x_{2}^{2}-x_{1}y &
g_4=x_{1}y^{a}&
g_5=x_{2}y^{a}&
g_6=y^{a+1}
\end{matrix}\Big)
$$

$$
\phi_{2}(b=1)=
\begin{pmatrix}
x_2 & y^{a} & -y & 0 & 0 & 0 & 0 & 0 \\
-x_1& 0 & x_2 & y^{a} & 0 & 0 & 0 & 0 \\
x_0  & 0 & -x_1 & 0 & y^{a} & 0 & 0 & 0 \\
0 & -x_1 & 0 & -x_2 &  0 & x_2 & y & 0 \\
0 & x_0 & 0  & 0  &  -x_2 & -x_1 & 0 & y \\
0 & 0 & 0 & x_0  &  x_1 & 0 &  -x_1&  -x_2 \\
\end{pmatrix},
$$

$$
\phi_{3}(b=1)=
\begin{pmatrix}
y^{a} & 0 & 0 \\
-x_2& y & 0 \\
0 & y^{a} & 0  \\
 x_1 & -x_2 & 0 \\
 -x_0 & x_1 & 0 \\
 0 & -x_2 &  y \\
  0 & x_1 &  -x_2 \\
 0 & -x_0 &  x_1 \\
\end{pmatrix},
$$

}

{\footnotesize

$$
\phi_{1}(b=2)=\Big(
\begin{matrix}
g_1=x_{1}^{2}-x_{0}x_{2} &
g_2=x_{1}x_{2}-x_{0}y &
g_3=x_{2}^{2}-x_{1}y &
g_4=x_{2}y^{a}&
g_5=y^{a+1}
\end{matrix}\Big)
$$

$$
\phi_{2}(b=2)=
\begin{pmatrix}
x_2  & -y & 0 & 0 & 0  \\
-x_1&  x_2 & y^{a} & 0 & 0 \\
x_0  & -x_1 & 0 & y^{a} & 0  \\
0 & 0 &  -x_1 & -x_2 &  y  \\
0 & 0 &   x_0 & x_1 &  -x_2 \\
\end{pmatrix},
$$

$$
\phi_{3}(b=2)=
\begin{pmatrix}
g_{5}=y^{a+1}  \\
g_{4}=x_2y^{a}  \\
-g_{3}=-x_{2}^{2}+x_{1}y  \\
g_{2}= x_{1} x_{2}-x_{0}y \\
g_{1}=x_{1}^{2}-x_{0}x_{2} \\
\end{pmatrix},
$$

}

{\footnotesize

$$
\phi_{1}(b=3)=\Big(
\begin{matrix}
g_1=x_{1}^{2}-x_{0}x_{2} &
g_2=x_{1}x_{2}-x_{0}y &
g_3=x_{2}^{2}-x_{1}y &
g_4=y^{a+1}
\end{matrix}\Big)
$$

$$
\phi_{2}(b=3)=
\begin{pmatrix}
x_2 & y^{a+1} & -y & 0 & 0  \\
-x_1& 0 & x_2 & y^{a+1} & 0  \\
x_0  & 0 & -x_1 & 0 & y^{a+1}  \\
0 & -x_{1}^{2}+ x_{0}x_{2} & 0 & -x_{1}x_{2}+x_{0}y & - x_{2}^{2}+ x_{1}y  \\
\end{pmatrix},
$$

$$
\phi_{3}(b=3)=
\begin{pmatrix}
y^{a+1} & 0  \\
-x_2 & y  \\
0 & y^{a+1}   \\
 x_1 & -x_2  \\
 -x_0 & x_1 \\
\end{pmatrix},
$$

}

\end{theorem}

\begin{proof} We will prove the theorem for the three cases, $b=$ 1, 2, and 3.\\

\noindent {\bf Case \textit{b} $=$ 1.}\\
It is easy to show that $\phi_{1}(1)\phi_{2}(1)=\phi_{2}(1)\phi_{3}(1)=0$
which proves that the above sequence is a complex.
To prove the exactness, we use Corollary 2 of Buchsbaum-Eisenbud theorem in \cite{eisenbud-buchsbaum}.
We have to show that rank $\phi_{1}(1)=1$, rank $\phi_{2}(1)=5$ and  rank $\phi_{3}(1)=3$, and also that 
$I(\phi_{i}(1))$ contains a regular sequence of length $i$ for all $1 \leq i \leq 3$. rank $\phi_{1}(1)=1$ is trivial.
We want to show that rank $\phi_{2}(1)=5$. Since the columns of the matrix $\phi_{2}(1)$ are related by the 
generators of the defining ideal $I(C)$,
note that all $6 \times 6$ minors of  $\phi_{2}(1)$ are zero.  $\phi_{2}(1)$ has a non zero divisor in the kernel.
By McCoy's theorem rank $\phi_{2}(1) \leq 5$. The determinants of $5 \times 5$ minors of $\phi_{2}(1)$ are
$x_{0}g_{6}^{2}$ when the 6th row and the columns 3, 5 and 6 are deleted, and $x_{1}g_{2}^{2}$ when
the 2nd row and the columns 2, 5 and 8 are deleted. Since $\{x_{0}g_{6}^{2}, x_{1}g_{2}^{2}\}$ are relatively prime,  $I(\phi_{2}(1))$ contains  a regular sequence of length 2.
Also, among the $3 \times 3$ minors of $\phi_{3}(1)$, we have $\{-x_{0}g_{1}, -x_{1}g_{2},-x_{2}g_{3}\}$. They are relatively prime, so $I(\phi_{3}(1))$ contains a regular sequence of length 3.\\

\noindent {\bf Case \textit{b} $=$ 2.}\\
Clearly $\phi_{1}(2)\phi_{2}(2)=\phi_{2}(2)\phi_{3}(2)=0$ and rank $\phi_{1}(2)=1$ and  rank $\phi_{3}(2)=1$.
We have to show that  rank $\phi_{2}(2)=4$ and $I(\phi_{i}(1))$ contains a regular sequence of length $i$ for all $1 \leq i \leq 3$. Among the $4 \times 4$ minors of $\phi_{2}(2)$, $I(\phi_{2}(2))$ contains 
$\{-g_{1}^2, -g_{2}^2\}$.  These two determinants constitute a regular sequence of length 2, since they are relatively prime. \\

\noindent {\bf Case \textit{b} $=$ 3.}\\
As in the previous cases, we have to show that rank $\phi_{1}(3)=1$, rank $\phi_{2}(3)=3$ and  rank $\phi_{3}(3)=2$, and also that 
$I(\phi_{i}(1))$ contains a regular sequence of length $i$ for all $1 \leq i \leq 3$. rank $\phi_{1}(3)=1$ is trivial.
We have to show that rank $\phi_{2}(3)=3$.  $\phi_{2}(3)$ has a non zero divisor in the kernel.
By McCoy's theorem rank $\phi_{2}(3) \leq 3$. Among the $3 \times 3$ minors of $\phi_{2}(3)$, $I(\phi_{2}(3))$ contains 
$\{g_{1}^2, g_{2}^2\}$ which is a regular sequence of length 2, since they are relatively prime. 
Also, among the $2 \times 2$ minors of $\phi_{3}(3)$, we have $\{g_{1}, -g_{2}, g_{3}\}$. They are 
relatively prime, so $I(\phi_{3}(3))$ contains a regular sequence of length 3.

\end{proof}

\begin{corollary} Under the hypothesis of Theorem 3.1., the minimal graded free resolution of the associated graded ring $G$ is given by

if $b=$1
\begin{center}
$0 \longrightarrow  R(-(a+3))^{3}  \xrightarrow{\phi_{3}(b)}  
R(-3)^{2}  \bigoplus R(-(a+2))^{6} \xrightarrow{\phi_{2}(b)} R(-2)^{3}\bigoplus R(-(a+1))^{3}  
\xrightarrow{\phi_{1}(b)} R $
\end{center}

if $b=$2
\begin{center}
$0 \longrightarrow  R(-(a+4)) \xrightarrow{\phi_{3}(b)}  R(-3)^2 \bigoplus R(-(a+2))^{3}    \xrightarrow{\phi_{2}(b)} R(-2)^{3}\bigoplus R(-(a+1))^{2}  
\xrightarrow{\phi_{1}(b)} R $
\end{center}

if $b=$3
\begin{center}
$0 \longrightarrow  R(-(a+4))^{2}  \xrightarrow{\phi_{3}(b)}  R(-3)^{2}  \bigoplus  R(-(a+3))^{3}  \xrightarrow{\phi_{2}(b)} R(-2)^{3}\bigoplus R(-(a+1))  
\xrightarrow{\phi_{1}(b)} R $
\end{center}

\end{corollary}
\begin{flushright}
$\Box$
\end{flushright}
If $H_{G}(i)=dim_{k}(m^{i}/m^{i+1})$ is the Hilbert function of $G$, then 

\begin{corollary} Under the hypothesis of Theorem 3.1., the Hilbert function of the associated graded ring $G$ is given by\\

if $b=$ 1\\

\noindent $H_{G}(i)= \left(\!\!\!\begin{array}{c} i+3 \\ 3 \end{array}\!\!\! \right)
-3\left(\! \!\! \begin{array}{c} i+1 \\ 3 \end{array}\!\!\! \right)
-3\left(\!\!\! \begin{array}{c} i-a+2 \\ 3\end{array}\!\!\! \right)
+2 \left(\!\!\! \begin{array}{c} i \\ 3 \end{array}\!\!\! \right)
+6\left(\!\!\! \begin{array}{c} i-a+1 \\ 3 \end{array}\!\!\! \right)
-3\left(\!\!\! \begin{array}{c} i-a \\ 3 \end{array}\!\!\! \right)$\\

if $b=$ 2\\

\noindent $H_{G}(i)= \left(\!\!\! \begin{array}{c} i+3 \\ 3 \end{array}\!\!\! \right)
-3\left(\!\!\! \begin{array}{c} i+1 \\ 3 \end{array}\!\!\! \right)
-2\left(\!\!\! \begin{array}{c} i-a+2 \\ 3\end{array}\!\!\! \right)
+2 \left(\!\!\! \begin{array}{c} i \\ 3 \end{array}\!\!\! \right)
+3\left(\!\!\! \begin{array}{c} i-a+1 \\ 3 \end{array}\!\!\! \right)
-\left(\!\!\! \begin{array}{c} i-a-1 \\ 3 \end{array}\!\!\! \right)$\\

if b=3\\

\noindent $H_{G}(i)= \left(\!\!\! \begin{array}{c} i+3 \\ 3 \end{array}\!\!\! \right)
-3\left(\!\!\! \begin{array}{c} i+1 \\ 3 \end{array}\!\!\! \right)
-\left(\!\!\! \begin{array}{c} i-a+2 \\ 3\end{array}\!\!\! \right)
+2\left(\!\!\! \begin{array}{c} i \\ 3 \end{array}\!\!\! \right)
+3\left(\!\!\! \begin{array}{c} i-a \\ 3 \end{array}\!\!\! \right)
-2\left(\!\!\! \begin{array}{c} i-a-1 \\ 3 \end{array}\!\!\! \right)$

\end{corollary}
\vspace{5mm}

\bibliographystyle{amsplain}

\end{document}